\documentclass[12pt,leqno]{amsart}
\usepackage{amssymb,amsthm,amsmath,latexsym}
\usepackage{tikz-cd}
\usepackage{array}
\newtheorem{theorem}{\sc Theorem}[section]

\newtheorem{lem}[theorem]{\sc Lemma}
\newtheorem{prop}[theorem]{\sc Proposition}

\newtheorem{ex}[theorem]{\sc Example}
\newtheorem{defi}[theorem]{\sc Definition}

 \newtheorem*{thmA}{Theorem A}
 \newtheorem*{thmB}{Theorem B}
 \newtheorem*{thmC}{Theorem C}

 \DeclareMathOperator{\Syl}{Syl}
 \DeclareMathOperator{\Aut}{Aut}

 \DeclareMathOperator{\lcm}{lcm}

\title[Automorphism Orbits]{Some families of locally graded groups with finitely many orbits under automorphisms}
\author[Dantas]{Alex C. Dantas}
\author[Sousa]{Juc\'elia F. de Sousa}
\address{Departamento de Matem\'atica, Universidade de Bras\'ilia,
Brasilia-DF, 70910-900 Brazil}
\email{(Dantas) alexcdan@gmail.br}\email{(Sousa)  ogirsm@yahoo.com.br}

\begin{document}

\maketitle

\begin{abstract}
    In this work, we study three families of locally graded groups with finitely many orbits under automorphisms. We prove that: (i) a residually finite group with finitely many orbits under automorphisms is locally finite and has finite exponent; (ii) a finitely generated locally graded group with finitely many orbits under automorphisms is finite; and (iii) the Mal'cev $\mathbb{Q}$-completion of an $r$-generated free nilpotent group of class $c$ has finitely many orbits under automorphisms if and only if either $r = 2$ and $c = 3$, or $c \leq 2$.   
\end{abstract}

\section{Introduction}

A group $G$ is called locally graded if each finitely generated subgroup of $G$ has a proper subgroup of finite index. It is not difficult to see that the class of locally graded groups is closed under taking subgroups, extensions, and Cartesian products. Locally finite groups, residually finite groups, and linear groups are examples of locally graded groups. The question “Which locally graded groups have finitely many orbits under automorphisms?” is a major open problem. In recent decades many families of locally graded groups have been studied via their automorphism orbits. For example, the following locally graded groups with finitely many automorphism orbits have been considered: locally finite groups in \cite{N_0}, locally compact groups in \cite{S1, S2}, abelian and abelian-by-finite groups in \cite{J, MS2}, FC-groups in \cite{BD}, and virtually nilpotent groups in \cite{BDdM}. 

In this paper we treat three families of locally graded groups with finitely many automorphism orbits: residually finite groups, finitely generated locally graded groups, and the Mal'cev $\mathbb{Q}$-completion of a finitely generated free nilpotent group.  

In the first main result of this paper, we investigate residually finite groups that have finitely many automorphism orbits.

\begin{thmA} \label{thm:rf}
Let $G$ be a residually finite group with finitely many orbits under automorphisms. Then $G$ has finite exponent. Moreover, $G$ is locally finite. 
\end{thmA}

Note that a finite exponent abelian group is residually finite and has finitely many automorphism orbits; in contrast, a direct product of infinitely many copies of a non-abelian finite group does not (see \cite[Theorem A]{BD}). By the theorem above, a residually finite group with finitely many automorphism orbits embeds as a subgroup of a Cartesian product of copies of a finite group. Likely the first example of a non-abelian infinite residually finite group with finitely many automorphism orbits arose in the study of \emph{$\aleph_{0}$-categorical groups}. A countable group $G$ is $\aleph_{0}$-categorical if and only if $Aut(G)$ has finitely many orbits in its action on $G^n$ for each $n \geq 1$ (for details on $\aleph_{0}$-categorical groups, see \cite{N_0, R}). In \cite{N_0}, Apps characterized the countable non-abelian characteristically simple residually finite $\aleph_{0}$-categorical groups by showing they are \emph{Boolean powers} of finite simple groups, see Section 2 for a definition of Boolean power. On this topic, we prove the following.

\begin{prop}
 Let $G$ be a group and let $G_0 = 1$, $G_1$, $G_2$, ..., $G_{n-1}$, $G_{n} = G$ be characteristc  subgroups of $G$ such that
    $$1 = G_0 < G_{1} < G_{2} < \dots G_{n-1} < G_{n} = G$$
    and $G_{i + 1}/G_{i}$ is characteristically simple and residually finite for all $i = 0, 1, \dots, n-1$. If $\omega(G) \leq 7$, then $G$ is a soluble-by-finite group.   
\end{prop}

In Section 2, we use Boolean powers of a non-abelian simple group and give an example of a non-soluble residually finite group with eight orbits under automorphisms. To construct this example, we used the following proposition.

\begin{prop}
Consider $\Gamma = G^{\mathcal{R}}$ a Boolean power of a group $G$. Suppose that $G$ has $r + 1$ automorphism orbits and that $\mathcal{R}$ has $m + 1$ ring-automorphism orbits. Then $G^{\mathcal{R}}$ has at most
\begin{center}
    $1 + m C^r_1 + m^2 C^r_2 + \cdots + m^{r} C^r_r$
\end{center}
automorphism orbits, where $C^r_k = \binom{r}{k}$.
\end{prop}

Note that the finitely generated group $S_{3} \wr \mathbb{Z}$ is locally graded but not residually finite. In the second main result of this paper, we study finitely generated locally graded groups with finitely many orbits under automorphisms.

\begin{thmB}
Let $d$ and $n$ be positive integers. 
Let $G$ be a $d$-generator locally graded group with $n$ automorphism orbits. Then $G$ is finite. Moreover, if
$\pi$ is a finite set of primes such that $|G|$ is a $\pi$-number, then $|G|$ is $\{d,n,\pi\}$-bounded.
\end{thmB}

As usual, the expression ``$\{a,b, \ldots\}$-bounded'' means ``bounded  above by some function which depends only on the parameters $a,b, \ldots$''.

It is well known that a group isomorphic to a finite-dimensional $\mathbb{Q}$-vector space has two orbits under its automorphism group. Such a group is isomorphic to the Mal'cev $\mathbb{Q}$-completion of a free abelian group of finite rank. Thus, in the third main result of this paper we study the Mal'cev $\mathbb{Q}$-completion of the $r$-generated free nilpotent group $N_{r, c}$ of nilpotent class $c$.

\begin{thmC}
The Mal'cev $\mathbb{Q}$-completion of the free nilpotent $r$-generated group $N_{r, c}$ of class $c$ has finitely many orbits under automorphisms if, and only if either $c \leq 2$, or $r = 2$ and $c = 3$. 
\end{thmC}

\section{Residually Finite Groups with finitely many orbits under automorphisms}

\subsection{The proof of Theorem A}

\begin{thmA}
Let $G$ be a residually finite group with finitely many orbits under automorphisms. Then $G$ has finite exponent. Moreover, $G$ is locally finite. 
\end{thmA}

\begin{proof}
Firstly, we prove that $G$ is a torsion group. Suppose there exists $g$ in $G$ with infinite order. Since $G$ is residually finite, there exists a normal subgroup $N_{1}$ of $G$ such that $g \notin N_{1}$ and the index $[G : N_{1}] = n_{1}$ is finite. Thus 
$$1 < \langle g^{n_1} \rangle  \leq G^{n_{1}} \leq N_{1} < G.$$ 
Since $N_{1}$ is residually finite, there exists a normal subgroup $N_{2}$ of $N_{1}$ such that $g^{n_1} \notin N_{2}$ and the index $[N_{1} : N_{2}] = n_{2}$ is finite. Thus 
$$1 < \langle g^{n_1 n_{2}} \rangle  \leq (G^{n_{1}})^{n_2} \leq N_{2} < N_{1} < G.$$ 
Note that $(G^{n_{1}})^{n_2}$ is a proper subgroup of $G^{n_{1}}$. Continuing with this procedure we obtain an infinite chain
$$\dots < ( \dots ((G^{n_1})^{n_2}) \dots)^{n_s} < \dots < (G^{n_{1}})^{n_2} < G^{n_{1}} < G$$
of nontrivial proper characteristc subgroups of $G$, and since the group $G$ has finitely many orbits under automorfisms we have a contradiction.

Now, let $g_{1}, \dots, g_{n}$ be elements of $G$ such that $G = \bigcup_{i = 1}^{n} g_{i}^{Aut(G)}$. Then $|g| \in \{|g_{1}|, \dots, |g_{n}|\}$ for all $g \in G$. By Zelmanov's solution of the Restricted Burnside Problem, the group $G$ is locally finite, which completes the proof. 
\end{proof}

A group $G$ is said to be an $S$-space, for some finite non-abelian simple group $S$, if each finite subset of $G$ is contained in a subgroup of $G$ isomorphic to $S^n$, for some $n \geq 0$.

\begin{lem}(\cite[1983]{N_0}) \label{Wilson}
Suppose that $G$ is a characteristically simple non-abelian group containing a proper subgroup of finite index. Then $G$ is an $S$-space for
some finite simple non-abelian group $S$.
\end{lem}

\begin{prop} \label{2.2}
    Let $G$ be a group and let $G_0 = 1$, $G_1$, $G_2$, ..., $G_{n-1}$, $G_{n} = G$ be characteristc  subgroups of $G$ such that
    $$1 = G_0 < G_{1} < G_{2} < \dots G_{n-1} < G_{n} = G$$
    and $G_{i + 1}/G_{i}$ is characteristically simple and residually finite for all $i = 0, 1, \dots, n-1$. If $\omega(G) \leq 7$, then $G$ is a soluble-by-finite group.
\end{prop}

\begin{proof}
    We can suppose that $G$ is a non-abelian characterstically simple group. By Lemma \ref{Wilson}, $G$ is an infinite $S$-space, where $S$ is a non-abelian finite simple group. Then there is a finite subgroup $H$ of $G$ isomorphic to $S \times S \times S$. There exist three distinct primes $p$, $q$, and $r$ such that $\{1, p, q, r, pq, pr, qr, pqr\} \subset spec(G)$, thus $|spec(G)| \geq |spec(H)| \geq 8$, and $\omega(G) \geq 8$, a contradiction. 
\end{proof}

In the next section, we present examples of residually finite characteristically simple non-abelian groups with finitely many orbits under automorphisms; among these examples, there is one with $8$ orbits.

\subsection{Examples}

Before presenting the examples, we introduce the concept of a Boolean power. For further details, see \cite{BP}. A \emph{Boolean ring} $R$ is a commutative ring in which every element satisfies the identity
$x^2 = x \cdot x = x, \quad \text{for all } x \in R$.
In a Boolean ring $R$, we have $x + x = 0$ for every $x \in R$. Indeed,
$2x = (2x)^2 = (2x)(2x) = 4x^2 = 4x$,
and hence $x + x = 2x = 0$.

Let $G$ be a group, let $X$ be a nonempty set, and let $\mathcal{P}(X)$ be the power set of $X$. The set $\mathcal{P}(X)$ equipped with the operations symmetric union $A \triangle B = (A \setminus B) \cup (B \setminus A)$ and intersection $A \cap B$, for all $A$ and $B$, is a Boolean ring. Let $\mathcal{R}$ be a subring of the Boolean ring $\mathcal{P}(X)$. Consider the group $\prod_X G$. For $g \in G$ and $A \in \mathcal{R}$, define $
g_A \in \prod_X G$ by
\[ g_A = (g_x)_{x \in X}, \text{ where }
g_x =
\begin{cases}
g, & \text{if } x \in A, \\
e, & \text{otherwise}.
\end{cases}
\]
Under these conditions, we define the \textbf{Boolean power} of $G$ by $\mathcal{R}$, denoted by $G^{\mathcal{R}}$, 
 as the subgroup of $\prod_X G$ generated by all elements of the form 
$g_A$, with $g \in G$ and $A \in \mathcal{R}$; that is,
\[
G^{\mathcal{R}} = \langle g_A \mid g \in G,\; A \in \mathcal{R} \rangle.
\]
It follows that each $\gamma \in G^\mathcal{R}$ can be written uniquely (up to the order of the factors) in the form
$$(g_1)_{A_1} \cdots (g_n)_{A_n},$$
where the elements $g_i$ are distinct non-trivial elements of $G$, and the sets $A_i$ are pairwise disjoint non-empty subsets of $X$ that are elements of $\mathcal{R}$. It is worth noting that if $X$ is any set and $\mathcal{R}$ is the subring of finite subsets of $X$, then $G^\mathcal{R}$ is the direct product of $|X|$ copies of $G$. Thus, Boolean powers extend the concept of the direct product of a group $G$.

For our purposes, we now describe some automorphisms of a Boolean power $\Gamma = G^{\mathcal{R}}$. 

\begin{itemize}
    \item[1.] Let $\alpha$ be an automorphism of the ring $\mathcal{R}$. Then $\alpha$ extends to an automorphism $\alpha^\Gamma$ of $\Gamma$ by defining $\alpha^\Gamma$ on the generators of $\Gamma$ by
    \[
    (g_A)^{\alpha^\Gamma} = g_{A^\alpha} \quad \text{for all } g \in G, \, A \in \mathcal{R}.
    \]
    \noindent We write $(\operatorname{Aut}(\mathcal{R}))^{\Gamma} = \langle \alpha^\Gamma \mid \alpha \in \operatorname{Aut}(\mathcal{R}) \rangle$ for the subgroup of $\operatorname{Aut}(\Gamma)$ isomorphic to $\operatorname{Aut}(\mathcal{R})$;
    
    \item[2.] Let $f$ be an automorphism of $G$. Then $f$ extends to an automorphism $f_X$ of $\Gamma$ by defining $f_X$ on the generators of $\Gamma$ by
    $$
    (g_A)^{f_X} = (g^f)_A \quad \text{for all } g \in G \text{ and for all } A \in \mathcal{R};
    $$
    
    \item[3.] Given $A \in \mathcal{R}$ and $f \in \operatorname{Aut}(G)$, define $f_A$ to be the automorphism of $\Gamma$ given on the generators of $\Gamma$ by
    $$
    (g_B)^{f_A} = (g^f)_{A \cap B} (g)_{B \setminus A}, 
    \quad \text{for all } g \in G \text{ and for all } B \in \mathcal{R}.
    $$
    
    \noindent We write $\mathcal{F}_{\mathcal{R}} = \langle f_X, f_A \mid A \in \mathcal{R} \rangle \leqslant \operatorname{Aut}(\Gamma)$.
\end{itemize}

\begin{prop}(\cite[Corollary C1]{BP})
The group $(\operatorname{Aut}(\mathcal{R}))^\Gamma$ normalizes $\mathcal{F}_{\mathcal{R}}$. In particular, the product 
$\mathcal{F}_{\mathcal{R}} (\operatorname{Aut}(\mathcal{R}))^\Gamma$ is a subgroup of $\operatorname{Aut}(\Gamma)$ 
isomorphic to $\mathcal{F}_{\mathcal{R}} \rtimes (\operatorname{Aut}(\mathcal{R}))^\Gamma$.
\end{prop}

\begin{proof}
Indeed, let $\alpha^\Gamma \in (\operatorname{Aut}(\mathcal{R}))^\Gamma$ and let $f_A, f_X$ be generators of $\mathcal{F}_{\mathcal{R}}$. Then
\[
(\alpha^\Gamma)^{-1} f_A (\alpha^\Gamma) = f_{A^\alpha}
\quad \text{and} \quad
(\alpha^\Gamma)^{-1} f_X (\alpha^\Gamma) = f_X \in \mathcal{F}_{\mathcal{R}},
\]
since for every generator $g_B$ of $\Gamma$ we have
\[
g_B \overset{(\alpha^\Gamma)^{-1}}{\longmapsto} g_{B^{\alpha^{-1}}}
\overset{f_A}{\longmapsto} (g^f)_{B^{\alpha^{-1}} \cap A} \, g_{B^{\alpha^{-1}} \setminus A}
\overset{\alpha^\Gamma}{\longmapsto} (g^f)_{B \cap A^\alpha} \, g_{B \setminus A^\alpha}
= (g_B)^{f_{A^\alpha}},
\]
and
\[
g_B \overset{(\alpha^\Gamma)^{-1}}{\longmapsto} g_{B^{\alpha^{-1}}}
\overset{f_X}{\longmapsto} (g^f)_{B^{\alpha^{-1}}}
\overset{\alpha^\Gamma}{\longmapsto} (g^f)_B.
\]
Thus,
\[
\mathcal{A} = \mathcal{F}_{\mathcal{R}} (\operatorname{Aut}(\mathcal{R}))^\Gamma
\simeq \mathcal{F}_{\mathcal{R}} \rtimes (\operatorname{Aut}(\mathcal{R}))^\Gamma.
\]
\end{proof}

\begin{prop}  (\cite[Corollary C1]{BP}) \label{prop3.2}
Let $\Gamma = G^\mathcal{R}$ be a Boolean power of $G$, where $G$ is a finite non-abelian indecomposable group such that $\operatorname{Hom}(G, Z(G)) = 1$. Then $\operatorname{Aut}(\Gamma)$ is isomorphic to $\mathcal{F}_{\mathcal{R}} \rtimes (\operatorname{Aut}(\mathcal{R}))^\Gamma$.
\end{prop}

The following proposition provides an upper bound for the number of automorphism orbits of a Boolean power.
\begin{prop}\label{prop3.3}
Consider $\Gamma = G^{\mathcal{R}}$ a Boolean power of a group $G$. Suppose that $G$ has $r + 1$ automorphism orbits and that $\mathcal{R}$ has $m + 1$ ring-automorphism orbits. Then $G^{\mathcal{R}}$ has at most
\begin{center}
    $1 + m C^r_1 + m^2 C^r_2 + \cdots + m^{r} C^r_r$
\end{center}
automorphism orbits, where $C^r_k = \binom{r}{k}$.

\end{prop}

\begin{proof}
By hyppothesis there exist $g_0 = e, g_1, ..., g_r \in G$ and $A_0 = \emptyset, A_1, ..., A_m \in \mathcal{R}$ such that $G = \dot{\cup}_{j=0}^{r} g_j^{Aut(G)}$ and $\mathcal{R} = \dot{\cup}_{i=0}^{m} A_i^{Aut(\mathcal{R})}$. Let ${\bf g} \in \Gamma$ be a nontrivial element. Note that the element ${\bf g}$ can be written in the form
\begin{center}
    ${\bf g} = (h_1)_{B_1}(h_2)_{B_2} \cdots (h_n)_{B_n},$
\end{center}
where $B_1, B_2, \dots, B_n$ are pairwise disjoint. Reordering if necessary, there exist
$i_1 < \cdots < i_{l-1} \in \{1,2,\dots,n\}$ with $l \le r$, 
$j_1, \dots, j_l \in \{1,2,\dots,r\}$, and 
$f_1, \dots, f_n \in \operatorname{Aut}(G)$ such that
\begin{center}
$h_{1} = g_{j_1}^{f_1}, \ h_2 = g_{j_1}^{f_2}, \dots, h_{i_1} = g_{j_1}^{f_{i_1}},$ \\[8pt]
$h_{i_1+1} = g_{j_2}^{f_{i_1+1}}, \dots, h_{i_2} = g_{j_2}^{f_{i_2}},$ \\[8pt]
$\cdots$ \\[4pt]
$h_{i_{l-1}+1} = g_{j_l}^{f_{i_{l-1}+1}}, \dots, h_n = g_{j_l}^{f_n}$.
\end{center}

Consider the automorphisms $(f_1^{-1})_{B_1}, \dots, (f_n^{-1})_{B_n}$ of $\Gamma$. Then
\[
f = (f_1^{-1})_{B_1} \cdots (f_n^{-1})_{B_n} \in \operatorname{Aut}(\Gamma)
\]
is such that
\[
{\bf g}^f
= (h_1^{f_1^{-1}})_{B_1} \cdots (h_n^{f_n^{-1}})_{B_n} =\]
\[= (g_{j_1})_{B_1} \cdots (g_{j_1})_{B_{i_1}}
(g_{j_2})_{B_{i_1+1}} \cdots (g_{j_2})_{B_{i_2}} \cdots (g_{j_l})_{B_{j_{l-1} + 1}} \cdots(g_{j_l})_{B_n}.
\]

Take $D_{k} = B_{i_{k-1} + 1} \dot{\cup} B_{i_{k-1} + 2} \dot{\cup} \dots \dot{\cup} B_{i_{k}}$. Note that if $\alpha \in Aut(\Gamma)$, then
\[
(D_{1} \dot{\cup} \cdots \dot{\cup} D_{l})^\alpha
= D_1^\alpha \dot{\cup} \cdots \dot{\cup} D_l^\alpha.
\]
Thus the $D_1, \dots, D_l$ may be fixed for each $l$ and each $g_{j_1}, \dots, g_{j_l}$. Hence there exists $\beta \in \operatorname{Aut}(\mathcal{R})$ such that
\[
{\bf g}^{f \alpha^\Gamma \beta^\Gamma}
= (g_{j_1})_{D_1} \cdots (g_{j_l})_{D_l}.
\]

Proceeding in this way, we obtain the following orbits: \\

\noindent
1. The trivial orbit; \\

\noindent
2. The orbits of elements ${\bf g}$ such that
$h_1, \dots, h_n \in g_i^{\operatorname{Aut}(G)}$,
yielding at most $m \binom{r}{1}$ orbits; \\

\noindent
3. The orbits of elements ${\bf g}$ such that
$$h_1, \dots, h_{i_1} \in g_{j_1}^{\operatorname{Aut}(G)}, \dots,
h_{i_{l-1}+1}, \dots, h_{i_l} \in g_{j_l}^{\operatorname{Aut}(G)}, \, l \geq 2,$$
yielding at most $m^l \binom{r}{l}$ orbits. \\

Therefore, the subgroup
$\mathcal{F}_{\mathcal{R}} \rtimes (\operatorname{Aut}(\mathcal{R}))^\Gamma$
of $\operatorname{Aut}(\Gamma)$ acts on $\Gamma$ with at most
\begin{center}
$1 + m \binom{r}{1} + m^2 \binom{r}{2} + \cdots + m^{r} \binom{r}{r}$
\end{center}
orbits.

\end{proof}
If $G$ is residually finite, then $\mathcal{G} = \prod_{X} G$ is also residually finite. In particular, any Boolean power $\Gamma = G^{\mathcal{R}}$ is residually finite whenever $G$ is residually finite. Next, we present some examples of infinite residually finite non-abelian groups that have finitely many orbits under automorphisms. For this purpose, we consider the following Boolean rings.

Let $X = [0,1) \subset \mathbb{R}$ endowed with the topology $\mathcal{T}$ generated by the basis
\[
\mathbf{b} = \{ [a,b) \mid a,b \in \mathbb{Q}, \, 0 \le a < b < 1 \}.
\]
Note that $(\mathcal{T}, \bigtriangleup, \cap)$ is a Boolean ring. Let $\mathcal{R}_1$ be the subring of $\mathcal{T}$ generated by $\mathbf{b}$. We also consider the subring $\mathcal{R}_2$ of $\mathcal{T}$ generated by
\[
\mathbf{b}_1 = \{ [a,b) \mid a,b \in \mathbb{Q}, \, 0 \le a < b \le 1 \}.
\]
Note that $\mathcal{R}_1$, $\mathcal{R}_2$, and 
\[
\mathcal{D}_0 = \{ A \in \mathcal{P}(\mathbb{Q}) \mid A \text{ is finite} \}
\]
are countable Boolean rings. We now analyze the action of the automorphism groups of these rings and describe their orbits. Consider the interval $[0,1)$ endowed with the topology $\mathcal{T}$ generated by the basis $\mathbf{b}$. Note that this topology coincides with the one generated by the basis $\mathbf{b}_1$.

Any homeomorphism $f : [0,1) \to [0,1)$ such that $A^f, A^{f^{-1}} \in \mathcal{R}_1$ for every $A \in \mathcal{R}_1$ defines a ring automorphism $\alpha_f$ of $\mathcal{R}_1$. Indeed, define $\alpha_f : \mathcal{R}_1 \to \mathcal{R}_1$ by $A^{\alpha_f} = A^f$. Then
    $$
    (A \bigtriangleup B)^{\alpha_f} = (A \bigtriangleup B)^f = ((A \setminus B) \cup (B \setminus A))^f =$$
    $$=(A \setminus B)^f \cup (B \setminus A)^f = A^f \bigtriangleup B^f = A^{\alpha_f} \bigtriangleup B^{\alpha_f},
    $$
    and
    $$
    (A \cap B)^{\alpha_f} = (A \cap B)^f = A^f \cap B^f = A^{\alpha_f} \cap B^{\alpha_f}.
    $$
    Hence, $\alpha_f$ is a ring homomorphism. Since $\emptyset^f = \emptyset$ and $A^{f^{-1}} \in \mathcal{R}_1$, it follows that $\alpha_f$ is injective and surjective, and therefore a ring automorphism.

    Now, given $A, B \in \mathcal{R}_1 \setminus \{\emptyset\}$, there exists a homeomorphism $f$ from $[0,1)$ to $[0,1)$ such that $A^f = B$. In fact, suppose that  
    $$A = [a_1, b_1) \cup \cdots \cup [a_n, b_n) \text{ and } B = [x_1, y_1) \cup \cdots \cup [x_m, y_m)$$ 
    with $0 \leq a_1 < b_1 < \dots < a_n < b_n < 1$ and $0 \leq x_1 < y_1 < \dots < x_m < y_m < 1$. We have twelve cases to consider: $a_1 = 0$ or $a_1 > 0$, $n < m$, $n = m$, or $n > m$, $x_1 = 0$ or $x_1 > 0$. Consider the case $a_1 = 0$, $n < m$ and $x_1 > 0$. Write
    $$[a_{n}, b_n) = [a_{n1}, b_{n1}) \dot{\cup} [a_{n2}, b_{n2}) \dot{\cup} \dots \dot{\cup} [a_{n(m-n+1)}, b_{n(m-n+1)})$$
    $$ \text{ with }a_{n1} = a_n, a_{n2} = b_{n1}, \dots, a_{n(m-n + 1)} = b_{n(m-n)}, b_{n(m-n+1)} = b_{n};$$
    
    $$[b_{n}, 1) = [c_{1}, d_{1}) \dot{\cup} [c_{2}, d_{2}) \dot{\cup} \dots \dot{\cup} [c_{m-n+1}, d_{m-n+1})$$
    $$ \text{ with }c_{1} = b_n, c_{2} = d_{1}, \dots, c_{m-n+1} = d_{m-n}, d_{m-n+1} = 1;$$
    
    $$[0, 1) \setminus A = [b_1, a_2) \dot{\cup} [b_2, a_3) \dot{\cup} \dots \dot{\cup}[b_{n-1}, a_n) \dot{\cup} [b_{n}, 1);$$ \\
    and
    
    $$[0, 1) \setminus B = [0, x_1) \dot{\cup} [y_{1}, x_2) \dot{\cup} \dots \dot{\cup}[y_{n-2}, x_{n-1}) \dot{\cup} [y_{n-1}, x_n) \dot{\cup} \dots$$
    $$\dots \dot{\cup} [y_{m-1}, x_{m})\dot{\cup} [y_{m}, 1).$$
 Since
    \begin{align*} [a, b) & \rightarrow [x, y) \\
    t &\mapsto \frac{x - y}{a - b}t + \frac{ay - bx}{a - b} \end{align*}
    is a homeomorphism, there is a homeomorphism $f: [0, 1) \rightarrow [0, 1)$ such that
    \begin{align*} [a_1 = 0, b_1)^f &= [x_1, y_1) \\ [a_2, b_2)^f &= [x_2, y_2) \\
    \vdots \\
    [a_{n-1}, b_{n-1})^f &= [x_{n-1}, y_{n-1}) \\
    [a_{n1}, b_{n1})^f &= [x_{n}, y_{n})\\
    [a_{n2}, b_{n2})^f &= [x_{n+1}, y_{n+1})\\
    \vdots \\
    [a_{n(m-n+1)}, b_{n(m-n+1)})^f &= [x_{m}, y_{m})\\   
    [b_1, a_2)^f &= [0, x_1)\\
    [b_2, a_3)^f &= [y_1, x_2)\\
    \vdots \\
    [b_{n-1}, a_n)^f &= [y_{n-2}, x_{n-1}) \\    
    [c_{1}, d_{1})^f &= [y_{n-1}, x_n)\\
    [c_{2}, d_{2})^f &= [y_{n}, x_{n+1})\\
    \vdots\\
    [c_{m-n}, d_{m-n})^f &= [y_{m-1}, x_m)\\
    [c_{m-n+1}, d_{m-n+1} &= 1)^f = [y_{m}, 1)
    \end{align*}
    As $f$ is a picewise linear hemeomorphism from $[0, 1)$ to $[0, 1)$ and $[c_{m-n+1}, 1)^f = [y_{m}, 1)$, we have $D^f \in \mathcal{R}_1$ for all $D \in \mathcal{R}_1$. By construction $A^f = B$. The other cases follow the same way. Therefore the ring automorphism group $\mathrm{Aut}(\mathcal{R}_1)$ acts with only two orbits on $\mathcal{R}_1$, that is,
    $$
    \mathcal{R}_1 = \{\emptyset\} \,\dot{\cup}\, [0,1/2)^{\mathrm{Aut}(\mathcal{R}_1)}.
    $$ 

Similarly, one can prove that
    $$
    \mathcal{R}_2 = \{\emptyset\} \,\dot{\cup}\, [0,1/2)^{\mathrm{Aut}(\mathcal{R}_2)} \,\dot{\cup}\, [0,1)^{\mathrm{Aut}(\mathcal{R}_2)}.
    $$
On the other hand, the ring $\mathcal{D}_0$ has infinitely many orbits under its ring automorphism group.

\begin{ex} ({\bf Infinite, countable, solvable, residually finite group that is not virtually abelian and has four automorphism orbits}) 
Consider the symmetric group $S_3 = \mathrm{Sym}(\{1, 2, 3\})$ of degree $3$ and let $\mathcal{R}_1$ be the Boolean ring defined above. Let $\Gamma_1 = S_3^{\mathcal{R}_1}$ be the Boolean power of $S_3$ by $\mathcal{R}_1$. Since the center $Z(S_3)$ of $S_3$ is trivial, by Propositions \ref{prop3.2} and \ref{prop3.3}, we have
\[
\omega(\Gamma_1) \leq 1 + 1 \cdot C^2_{1} + 1^2 \cdot C^{2}_{2}= 1 + 2 + 1 = 4.
\]
Since $spec(\Gamma_1) = \{1, 2, 3, 6\}$, it follows that $\omega(\Gamma_1) = 4$.
\end{ex}

\begin{ex} \label{Ex2.1.3} ({\bf Infinite, countable, residually finite group that is not virtually solvable and has eight automorphism orbits})  
Consider $A_5$, the alternating group of degree $5$, and let $\mathcal{R}_1$ be the Boolean ring defined above. Let $\Gamma_2 = A_5^{\mathcal{R}_1}$ be the Boolean power of $A_5$ by $\mathcal{R}_1$. Since the center $Z(A_{5})$ of $A_{5}$ is trivial, by Propositions \ref{prop3.2} and \ref{prop3.2}, we have
\[
\omega(\Gamma_2) \leq 1 + 1 \cdot C^3_{1} + 1^2 \cdot C^{3}_{2} + 1^3 \cdot C^{3}_{3}= 1 + 3 + 3 + 1 = 8.
\]
Since $spec(\Gamma_2) = \{1, 2, 3, 5, 6, 10, 15, 30\}$, it follows that $\omega(\Gamma_2) = 8$.
\end{ex}

\begin{ex} \label{exddd} ({\bf Infinite, uncountable, residually finite group that is not virtually solvable and has eight automorphism orbits})  
Let $\mathcal{R}$ be the Boolean ring generated by
\[
\mathbf{b}_2 = \{[a, b) \mid a, b \in \mathbb{R}, \text{ with } 0 \leq a < b < 1 \}.
\]
As in the case of the ring  $\mathcal{R}_1$, it follows that $\mathcal{R}$ has $2$ orbits under ring automorphisms. Similarly to Example \ref{Ex2.1.3}, the Boolean power $A_{5}^{\mathcal{R}}$ has $8$ automorphism orbits.
\end{ex}

\begin{ex} \label{BWR} ({\bf Boolean wreath product})  
Let $H$ be a subgroup of the symmetric group $Sym(\mathcal{R})$, where $\mathcal{R}$ is a Boolean subring of $\mathcal{P}(X)$ for some set $X$. Consider the semidirect product
\[
W = G^{\mathcal{R}} H,
\]
where
\[g_{A}^h = g_{A^h}, \]
$G^{\mathcal{R}} = \langle g_A \mid g \in G, A \in \mathcal{R} \rangle$ and $h \in H$.  We call $W$ the Boolean wreath product of $G$ by $H$. 

Suppose that $G$ and $\mathcal{R}$ have finitely many orbits under automorphisms and $H = (\mathcal{R}, +)$, then the Boolean wreath product $W = G^{\mathcal{R}} H$ has no more than $\omega(G^{R}) \omega(H)$ orbits, where
\[g_{A}^B = g_{A^B} = g_{A + B},\]
and a ring automorphis $\alpha$ extends for $W$ by
\[(g_A^{B})^{\alpha} = g_{A^{\alpha}}^{B^{\alpha}} = g_{A^{\alpha}+ B^{\alpha}} = g_{(A+B)^{\alpha}}.\]
\end{ex}



\section{Finitely generated locally graded groups with finitely many orbits under automorphisms}

In \cite{Osin}, Denis Osin proved that every countable torsion-free group can be embedded into a $2$-generated group with finitely many conjugacy classes. Thus, there exists a finitely generated group with finitely many automorphism orbits. In this section, we show that this does not occur in infinite locally graded groups. In particular, the group $S_3 \wr \mathbb{Z}$ does not have finitely many automorphism orbits.

\begin{thmB} \label{cor:rf}
Let $d$ and $n$ be positive integers. 
Let $G$ be a $d$-generator locally graded group with $\omega(G)=n$. Then $G$ is finite. Moreover, if
$\pi$ is a finite set of primes such that $|G|$ is a $\pi$-number, then $|G|$ is $\{d,n,\pi\}$-bounded.
\end{thmB}

\begin{proof}
(a) We argue by contradiction and suppose that $G$ is infinite. Since $G$ is a finitely generated locally graded group, it follows that there is a proper characteristic subgroup of finite index $K_1$ in $G$ (cf. \cite[7.2.9]{Hall}). As $G$ is infinite, we have  $K_1$ is also infinite. Moreover, the subgroup $K_1$ has at most $n-1$ automorphism orbits under the action of $\Aut(G)$. Moreover, $K_1$ is a finitely generated locally graded group. We now apply this argument again, with $G$ replaced by $K_1$, to obtain an infinite proper characteristic subgroup $K_2$ of finite index in $G$ which are contained in $K_1$. In particular, $\omega(K_2) \leqslant n-2$. We can obtain an infinite (proper) chain of characteristic subgroups of $G$, say $\{K_i\}_{i \in \mathbb{N}}$ such that $n = \omega(G) > \omega(K_1) > \omega(K_2) > \ldots > \omega(K_i) > \ldots \geqslant 1$. This contradicts our assumption. Thus, $G$ is finite. \\

\noindent (b) By item (a), the group $G$ is finite. It remains to prove that the order $|G|$ is $\{d,n,\pi\}$-bounded. 
We first prove that the exponent $\exp(G)$ is $\{n,\pi\}$-bounded. Set $p_1, \ldots, p_k$ be primes such that $\pi= \{p_1,\ldots,p_k\}$. We first observe that $\exp(G)$ divides $N$, where $N = p_1^n \cdot \ldots \cdot p_k^n$. Since $\omega(G)=n$, it follows that the exponent $\exp(P_i)$ divides $p_i^n$, for every $P_i \in \Syl_{p_i}(G)$. Now, we show that $\exp(G)$ divides $N$. Choose arbitrarily $x \in G$. It is clear that the $p_i$-part of $|x|$ divides $p_i^n$. In particular, $|x|$ divides $|N|$. As $\exp(G) = \lcm \   \{|x| : \ x \in G\}$ we have $\exp(G)$ divides $N$. By the positive solution of the RBP, there exists a positive integer $B(d,N)$ such that $|G| \leqslant B(d,N)$. Since $N$ is $\{n,\pi\}$-bounded, we deduce that $G$ is finite with $\{d,n,\pi\}$-bounded order, as well. The proof is complete.      
\end{proof}

\section{Orbits under automorphisms of Mal'cev $\mathbb{Q}$-completion of nilpotent free groups}

Let $N_{r,c}$ denote the free nilpotent group of rank $r$ and nilpotency class $c$. Note that the group $N_{r,c}$ is residually finite and torsion-free; by Theorem A, it does not have finitely many automorphism orbits. Let
$G = N_{r,c}^{\mathbb{Q}}$ be the Mal'cev $\mathbb{Q}$-completion of $N_{r,c}$. In this section, we prove that $G$ has finitely many automorphism orbits if and only if $G \cong N_{r,2}^{\mathbb{Q}}$ or $G \cong N_{2,3}^{\mathbb{Q}}$.

Before presenting the proofs, we recall the concepts and terminology that will be used in order to facilitate the understanding of the results.
\subsection{Basic commutador}
\begin{defi}
Let $G$ be a group generated by $X = \{x_1, \cdots, x_n\}$. 
A \emph{basic commutator} $b_j$ of weight $w(b_j)$ is defined as follows:\\
1. The elements of $X$ are the basic commutators of weight one. 
    We impose an arbitrary order on them and relabel them as 
    $b_1, b_2, \cdots, b_k$, where $b_i < b_j$ if $i < j$.\\
        2. Suppose that we have already defined and ordered the basic commutators 
    of weight less than $l > 1$. 
    The basic commutators of weight $l$ are the elements $[b_i, b_j]$, where:\\
   (i) $b_i$ and $b_j$ are basic commutators and        $w(b_i) + w(b_j) = l$,\\
        (ii) $b_i > b_j$, and\\
        (iii) if $b_i = [b_s, b_t]$, then $b_j \geq b_t$.\\
        3. The commutators of weight $l$ come after all commutators 
    of weight less than $l$ and are ordered arbitrarily among themselves.

The sequence $b_1, b_2, \cdots$ is called the \emph{basic sequence of basic commutators} on $X$.
\end{defi}

 \subsection{Mal'cev $\mathbb{Q}$-completion}     

\,

\,

For positive integers $r,c > 1$, we denote by  $N_{r,c} = G$  be a free    group of rank $r$ and   nilpotency  class $c$. 
Consider a polycyclic series of $ G$, of Hirsch length $n$,
\[
G = G_1 \rhd G_2 \rhd \cdots \rhd G_{n+1} = 1.
\]
Since each factor of this series is infinite cyclic, we may choose $u_i \in G_i$ such that
\[
G_i = \langle G_{i+1}, u_i \rangle, \quad i \in \{1, \cdots, n\}.
\]
Thus, every element of $G$ can be expressed in the form
\[
u_1^{\alpha_1} \cdots u_n^{\alpha_n},
\]
where $\alpha_1, \cdots, \alpha_n \in \mathbb{Z}$.
By definition, the set $\bar{u} = \{u_1, \cdots, u_n\}$ associated with the polycyclic series is called a \emph{Mal'cev basis} for $G$. 
The element $u_1^{\alpha_1} \cdots u_n^{\alpha_n}$ is said to have coordinates 
$(\alpha_1, \cdots, \alpha_n) \in \mathbb{Z}^n$ with respect to $\bar{u}$.
According to \cite[Theorem 4.8]{verde}, any basic sequence of basic commutators 
of weight at most $c$ forms a Mal'cev basis of the group $N_{r,c}$.

Let $\{u_1, \cdots, u_n\}$ be a Mal'cev basis of $N_{r,c}$ with respect to a polycyclic series. 
If $x = u_1^{\alpha_1} \cdots u_n^{\alpha_n}$ and 
$y = u_1^{\beta_1} \cdots u_n^{\beta_n}$ for some $\alpha_i, \beta_i \in \mathbb{Z}$, 
and if $\lambda \in \mathbb{Z}$, then
\[
xy = u_1^{f_1(\bar{\alpha}, \bar{\beta})} \cdots 
u_n^{f_n(\bar{\alpha}, \bar{\beta})}
\quad \text{and} \quad
x^\lambda = u_1^{g_1(\bar{\alpha}, \lambda)} \cdots 
u_n^{g_n(\bar{\alpha}, \lambda)}.
\]
The polynomials $f_i$ and $g_i$ are called the \emph{multiplication polynomials} 
and \emph{exponentiation polynomials} for $N_{r,c}$.

The Mal'cev $\mathbb{Q}$-completion of the group $N_{r,c}$ is the group defined by the set
\begin{center}
$N_{r,c}^{\mathbb{Q}} = \left\{ \bar{u}^{\bar{\beta}} = 
u_1^{\beta_1} \cdots u_n^{\beta_n} \mid \beta_i \in \mathbb{Q} \right\},$
\end{center}
equipped, respectively, with the multiplication and exponentiation by 
$\lambda \in \mathbb{Q}$ given by
\[
xy = u_1^{f_1(\bar{\alpha}, \bar{\beta})} \cdots 
u_n^{f_n(\bar{\alpha}, \bar{\beta})}
\]
and
\[
x^\lambda = u_1^{g_1(\bar{\alpha}, \lambda)} \cdots 
u_n^{g_n(\bar{\alpha}, \lambda)},
\]
where $x, y \in N_{r,c}^{\mathbb{Q}}$ are given by 
$x = u_1^{\alpha_1} \cdots u_n^{\alpha_n}$ and 
$y = u_1^{\beta_1} \cdots u_n^{\beta_n}$ with 
$\alpha_i, \beta_i \in \mathbb{Q}$.\\

For more details on basic commutators and the Mal'cev $\mathbb{Q}$-completion, see \cite{verde}.

\subsection{Width of the Mal'cev $\mathbb{Q}$-completion of $N_{r,c}$}

\,

\,

Let $G$ be a group. The width of an element $g$ in the commutator subgroup $G'$, 
denoted by $\lambda(g)$, is the smallest number $m$ such that $g$ can be expressed 
as a product of $m$ commutators of elements of $G$. 
The width $\lambda(G)$ of the group $G$ is defined as
\[
\lambda(G) = \max_{g \in G'} \lambda(g),
\]
whenever this maximum exists. Otherwise, we say that the width of $G$ is infinite. 
An element $g \in G$ such that $\lambda(g) = \lambda(G)$ is called an element of maximal width.

\begin{prop}[\cite{larguracomutator}] The following hold \\
    (a) The width of the group $N_{r,2}^{\mathbb{Q}}$ is equal to 
    $\left[\frac{r}{2} \right]$, for $r \geq 2$; \\        (b) The width $\lambda(N_{r,c})$ is equal to $r$, where 
    $r \geq 2$ and $c \geq 3$;\\
        (c) The width $\lambda(N_{r,c}^{\mathbb{Q}})$ is equal to $r$, 
    for $r \geq 2$, $c \geq 4$, as well as for $r \geq 3$, $c \geq 3$. 
    In the exceptional case, $\lambda(N_{2,3}^{\mathbb{Q}}) = 1$.\\
\end{prop}    
Here, the floor function, denoted by $[x]$, 
is the greatest integer less than or equal to $x$.

 To compute the number of automorphism orbits of $G$, we first require the following proposition, which establishes necessary and sufficient conditions for a map defined on the generators of $N_{r,c}$ to extend to an automorphism of $G$.

\begin{prop}\label{automorfismosNQ}
Let \( G = N_{r,c}^{\mathbb{Q}} = \langle x_1, x_2, \ldots, x_r \rangle^{\mathbb{Q}} \), where  
\( N_{r,c} = \langle x_1, x_2, \ldots, x_r \rangle \). Suppose that the elements  
\[
X_1 = x_1^{\alpha_{11}} \cdots x_r^{\alpha_{1r}} g_1,\ \dots,\ 
X_r = x_1^{\alpha_{r1}} \cdots x_r^{\alpha_{rr}} g_r,
\]
with \( g_1, \ldots, g_r \in G' \), are such that $
\{(\alpha_{11}, \ldots, \alpha_{1r}), \ldots, (\alpha_{r1}, \ldots, \alpha_{rr})\}
$
is a basis of the vector space \( \mathbb{Q}^r \) over \( \mathbb{Q} \). Then the map
 $$\varphi : x_1 \longmapsto X_1, x_2 \longmapsto X_2,
\ \cdots ,\
x_r\longmapsto X_r$$
extends to an automorphism of \( G \). Conversely, for every automorphism \( \varphi \) of \( G \), there exist elements  
\[
X_1 = x_1^{\alpha_{11}} \cdots x_r^{\alpha_{1r}} g_1,\ \dots,\ 
X_r = x_1^{\alpha_{r1}} \cdots x_r^{\alpha_{rr}} g_r,
\]
with \( g_1, \ldots, g_r \in G' \), such that  
$\{(\alpha_{11}, \ldots, \alpha_{1r}), \ldots, (\alpha_{r1}, \ldots, \alpha_{rr})\}$
is a basis of \( \mathbb{Q}^r \) and \( x_i^{\varphi} = X_i \) for \( i = 1, 2, \ldots, r \).
\end{prop}
\begin{proof}
    Note that the \(r\)-generated subgroup \(\langle X_1, \ldots, X_r \rangle\) of \(G\) has nilpotency class at most \(c\). Since \(N_{r,c}\) is free nilpotent, there exists a unique homomorphism
\[
\bar{\varphi} : N_{r,c} \longrightarrow \langle X_1, \ldots, X_r \rangle
\]
that completes the diagram

\[
\begin{tikzcd}
N_{r,c}
\arrow[dr, dashed, "\exists!\,\bar{\varphi}"]
& \\
\{x_1, \ldots, x_r\}
\arrow[u, hookrightarrow, "\iota"]
\arrow[r, "\varphi"]
& \langle X_1, \ldots, X_r \rangle
\end{tikzcd}
\]
Thus \(\bar{\varphi}\) extends to a homomorphism from \(G\) to \(G\). We prove that \(\varphi\) extends to an automorphism of \(G\) by induction on \(c\).
If \(c = 1\) then \(G\) is a \(\mathbb{Q}\)-vector space of dimension \(r\), and \(\varphi\) clearly extends to an automorphism of \(G\).
Assume \(c > 1\).
By induction the map \(\varphi\) induces an automorphism
\begin{align*}
\varphi_1 : G/Z(G) &\longrightarrow G/Z(G), \\
gZ(G) &\longmapsto g^{\bar{\varphi}}Z(G).
\end{align*}
Hence
$\langle X_1 Z(G), \ldots, X_r Z(G) \rangle^{\mathbb{Q}} = G/Z(G)$
and for each \(i = 1, 2, \ldots, r\) there exist \(q_{i1}, \ldots, q_{ir} \in \mathbb{Q}\) such that
\[
x_i Z(G) = X_1^{q_{i1}} \cdots X_r^{q_{ir}} y_i Z(G)
\qquad y_i \in G' \setminus Z(G).
\]
Now, let \(w(x_1, \ldots, x_r)\) be a basic commutator lying in \(Z(G)\). Then
\begin{align*}
w(x_1, \ldots, x_r)
&= w(X_1^{q_{11}} \cdots X_r^{q_{1r}} y_1, \ldots,
      X_1^{q_{r1}} \cdots X_r^{q_{rr}} y_r) \\
&= w(X_1^{q_{11}} \cdots X_r^{q_{1r}}, \ldots,
      X_1^{q_{r1}} \cdots X_r^{q_{rr}}),
\end{align*}
which belongs to
$\langle u \mid u \text{ is a basic commutator in } X_i,\ i = 1, \ldots, r \rangle^{\mathbb{Q}}.$

Consider the set
\[
U = \{\, u \mid u \text{ is a basic commutator in the } X_i \,\}.
\]
Then \(U\) is a generating set of the \(\mathbb{Q}\)-vector space \(Z(G)\).
Since \(|U| = \dim Z(G)\) it follows that \(U\) is a basis of \(Z(G)\).
Therefore
\[
\bar{\varphi} : w(x_1, \ldots, x_r) \longmapsto w(X_1, \ldots, X_r)
\]
extends to an isomorphism from \(Z(G)\) to \(Z(G)\), that is, to an automorphism of \(Z(G)\).

Note that
\begin{align*}
\bar{\varphi} : G &\longrightarrow G, \\
x_i &\longmapsto X_i, \qquad i = 1, 2, \ldots, r,
\end{align*}
is a well-defined homomorphism.
Let \(g \in G\).
Then \(g = hz\) with \(gZ(G) = hZ(G)\).
Since \(\varphi_1\) is surjective there exists \(kZ(G)\) such that
\[
(kZ(G))^{\varphi_1} = hZ(G),
\]
that is,
\[
k^{\bar{\varphi}} Z(G) = hZ(G).
\]
In particular, there exists \(z_1 \in Z(G)\) such that
\[
k^{\bar{\varphi}} z_1 = hz = g.
\]
Since \(z_1 \in Z(G)\) and \(\bar{\varphi}|_{Z(G)}\) is an automorphism,
there exists \(z_2 \in Z(G)\) such that \(z_1 = z_2^{\bar{\varphi}}\).
Hence,
\[
g = k^{\bar{\varphi}} z_1
  = k^{\bar{\varphi}} z_2^{\bar{\varphi}}
  = (k z_2)^{\bar{\varphi}}.
\]

Now suppose that \(g^{\bar{\varphi}} = 1\).
Then \(\bar{g}^{\varphi_1} = Z(G)\).
Hence
\[
\bar{g}^{\varphi_1} = g^{\bar{\varphi}} Z(G) = Z(G),
\]
which implies that \(g^{\bar{\varphi}} \in Z(G)\).
Thus \(g \in Z(G)\), and consequently \(g = 1\).

\vspace{0.3cm}

To conclude suppose that \(\varphi\) is an automorphism of \(G\).
Then
\begin{align*}
\bar{\varphi} : G/G' &\longrightarrow G/G', \\
gG' &\longmapsto g^{\varphi} G'
\end{align*}
is an automorphism.
Hence
\begin{align*}
(x_i G')^{\bar{\varphi}}
&= x_i^{\varphi} G' \\
&= x_1^{\alpha_{i1}} \cdots x_r^{\alpha_{ir}} G'.
\end{align*}
In particular for each \(i = 1, 2, \ldots, r\) there exists \(g_i \in G'\) such that
\[
x_i^{\varphi}
= x_1^{\alpha_{i1}} \cdots x_r^{\alpha_{ir}} g_i
= X_i g_i.
\]
That
\[
\{(\alpha_{11}, \ldots, \alpha_{1r}), \ldots, (\alpha_{r1}, \ldots, \alpha_{rr})\}
\]
is a basis of \(\mathbb{Q}^r\) follows from the fact that
\(\bar{\varphi}\) is an isomorphism of \(\mathbb{Q}\)-vector spaces.
\end{proof}
Below we present several applications of the preceding proposition. Let $\lambda(G)$ denote the width of the group $G$. For further details on the width of $G$, see \cite{larguracomutator}.

\begin{prop}\label{teorema4.1}
The group \( G = N_{r,2}^{\mathbb{Q}} \) has \( \lambda(G) + 2 \) automorphism orbits.
\end{prop}

\begin{proof}
Consider \( N_{r,2} = \langle x_1, x_2, \dots, x_r \rangle \). Thus, $G = \langle x_1, x_2, \dots, x_r \rangle^{\mathbb{Q}}$. 
Write \( r = 2k \) if \( r \) is even and \( r = 2k + 1 \) if it is odd.  
Note that
\[
G = \{1\} \,\dot{\cup}\, (Z(G)\setminus\{1\}) \,\dot{\cup}\, (G\setminus Z(G))
\]
and that \( Z(G) = G' \). Let
\[
g = x_1^{\alpha_{11}} \cdots x_r^{\alpha_{1r}} z \in G \setminus Z(G),
\]
where \( z \in Z(G) \) and \( (\alpha_{11},\ldots,\alpha_{1r}) \in \mathbb{Q}^r \setminus \{(0,\ldots,0)\} \).
Since the vector \( (\alpha_{11},\ldots,\alpha_{1r}) \) is nonzero, there exists a basis
\[
\{(\alpha_{11},\ldots,\alpha_{1r}),(\alpha_{21},\ldots,\alpha_{2r}),\ldots,(\alpha_{r1},\ldots,\alpha_{rr})\}
\]
of \( \mathbb{Q}^r \).
Define
\[
g_2 = x_1^{\alpha_{21}}\cdots x_r^{\alpha_{2r}},\ \ldots,\ 
g_r = x_1^{\alpha_{r1}}\cdots x_r^{\alpha_{rr}} \in G.
\]
By Proposition~\ref{automorfismosNQ}, the map
\begin{center}
  \begin{tabular}{>{\raggedleft}p{4cm} c p{5cm}}
   $\varphi: x_1 $&$ \longmapsto $&$ g$\\
   $x_2$ & $ \longmapsto$ &$ g_2$\\
   $x_3$ & $\longmapsto$ & $g_3$ \\ 
  $\vdots $ & $\vdots$ & $\vdots$ \\
  $x_r$ & $\longmapsto$ & $g_r$
\end{tabular}
\end{center}
extends to an automorphism of \( G \), and satisfies \( x_1^{\varphi} = g \).
Hence \( G \setminus Z(G) = x_1^{\Aut(G)} \).

We now prove \(Z(G)\setminus\{1\}\) is 
\[
[x_2,x_1]^{\Aut(G)}
\dot{\cup}
([x_2,x_1][x_4,x_3])^{\Aut(G)}
\dot{\cup} \cdots \dot{\cup}
([x_2,x_1][x_4,x_3]\cdots[x_{2k},x_{2k-1}])^{\Aut(G)}.
\]

Let \( z \in Z(G) \) with \( \lambda(z) = n \), where \( 1 \le n \le \lambda(G) \).
Then \( z \) can be written as
\[
z = \underbrace{[z_1,z_2][z_3,z_4]\cdots[z_{2n-1},z_{2n}]}_{n},
\]
where \( z_i = x_1^{\beta_{i1}}\cdots x_r^{\beta_{ir}} \) and
\( v_i = (\beta_{i1},\ldots,\beta_{ir}) \in \mathbb{Q}^r \).
We claim that the set
\[
\{v_1,v_2,\ldots,v_{2n}\}
\]
is linearly independent in the \( \mathbb{Q} \)-vector space \( \mathbb{Q}^r \).

Suppose, for a contradiction, that
\[
v_{2n} = \lambda_1 v_1 + \cdots + \lambda_{2n-1} v_{2n-1},
\qquad \lambda_i \in \mathbb{Q}.
\]
Then
\[
z_{2n} = z_1^{\lambda_1} z_2^{\lambda_2} \cdots z_{2n-1}^{\lambda_{2n-1}}.
\]
Assume, without loss of generality, that \( \lambda_i \neq 0 \) for
\( i = 1,2,\ldots,2n-1 \).
Thus,
\begin{align*}
z
&= [z_1,z_2]\cdots[z_{2n-1},z_1^{\lambda_1}\cdots z_{2n-1}^{\lambda_{2n-1}}] \\
\end{align*}
Hence,
\[
z = \underbrace{
[z_1 z_2^{\lambda_2/\lambda_1},\, z_2 z_{2n-1}^{-\lambda_1}]
\cdots
[z_{2n-3} z_{2n-2}^{\lambda_{2n-2}/\lambda_{2n-3}},\,
 z_{2n-2} z_{2n-1}^{-\lambda_{2n-3}}]
}_{n-1},
\]
which contradicts the assumption that \( \lambda(z) = n \).

Since the set \( \{v_1,\ldots,v_{2n}\} \) is linearly independent, we may add
vectors \( v_{2n+1},\ldots,v_r \) to complete it to a basis of \( \mathbb{Q}^r \).
Let
\[
z_j = x_1^{\beta_{j1}}\cdots x_r^{\beta_{jr}},
\qquad 2n+1 \le j \le r,
\]
with \( \beta_{jl} \in \mathbb{Q} \).
By Proposition~\ref{automorfismosNQ}, the map

\begin{center}
  \begin{tabular}{>{\raggedleft}p{4cm} c p{5cm}}
   $\varphi_{r}: x_1 $ & $ \longmapsto $ & $ z_1$\\
   $x_2$ & $ \longmapsto$ &$ z_2$\\
   $\vdots $ & $\vdots$ & $\vdots$ \\
   $x_{2n}$ & $ \longmapsto$ &$ z_{2n}$\\
  $\vdots $ & $\vdots$ & $\vdots$ \\
  $x_r$ & $\longmapsto$ & $z_r$
\end{tabular}
\end{center}

extends to an automorphism of \( G \), and satisfies
\[
([x_2,x_1][x_4,x_3]\cdots[x_{2n},x_{2n-1}])^{\varphi_r}
= [z_2,z_1][z_4,z_3]\cdots[z_{2n},z_{2n-1}].
\]
The result follows.
\end{proof}
\begin{prop}\label{propositionoo}
The group $G = N_{2,3}^{\mathbb{Q}}$ has $4$ automorphism orbits. 
\end{prop}

\begin{proof}
Let $N_{2,3} = \langle x_1, x_2 \rangle$, so that
$
G = \langle x_1, x_2 \rangle^{\mathbb{Q}}.
$
A sequence of basic commutators of $G$ is
$$
x_1,\; x_2,\; [x_2, x_1],\; [x_2, x_1, x_1],\; [x_2, x_1, x_2].
$$

We will prove that the automorphism orbits of $G$ are
$$
1^{\Aut(G)}, \quad x_1^{\Aut(G)} = G \setminus G', \quad
[x_2, x_1]^{\Aut(G)} = G' \setminus Z(G),$$
$$[x_2, x_1, x_1]^{\Aut(G)} = Z(G) \setminus \{1\}.
$$
Clearly, $1^{\Aut(G)} = \{1\}$. Let
$$
g_1 = x_1^{\alpha_1} x_2^{\alpha_2} f \in G \setminus G',
$$
where $f \in G'$ and $(\alpha_1, \alpha_2) \in \mathbb{Q}^2 \setminus \{(0,0)\}$.
Since the vector $(\alpha_1, \alpha_2) \in \mathbb{Q}^2$ is nonzero, there exists a basis
$$
\{(\alpha_1, \alpha_2), (\alpha_3, \alpha_4)\}
$$
of $\mathbb{Q}^2$.
Define $g_2 = x_1^{\alpha_3} x_2^{\alpha_4}$.
By Proposition~\ref{automorfismosNQ}, the map
\begin{center}
  \begin{tabular}{>{\raggedleft}p{4cm} c p{5cm}}
   $\varphi: x_1 $&$ \longmapsto $&$ g_1$\\
   $x_2$ & $ \longmapsto$ &$ g_2$
  \end{tabular}
\end{center}
extends to an automorphism of $G$ such that $x_1^{\varphi} = g_1$.
Thus, $$G \setminus G' = x_1^{\Aut(G)}.$$

Let $h \in G' \setminus Z(G)$. Then $h$ can be written as
$$
h =
[x_2, x_1]^{\alpha_1}
[x_2, x_1, x_1]^{\alpha_2}
[x_2, x_1, x_2]^{\alpha_3},
$$
with $\alpha_1 \neq 0$.
Note that there are four cases to consider: $\alpha_1 \neq 0, \  \alpha_1 \neq 0 \text{ and } \alpha_2 \neq 0, \ \alpha_1 \neq 0 \text{ and } \alpha_3 \neq 0, \ \alpha_1 \neq 0 \text{ and } \alpha_2, \alpha_3 \neq 0$. By Proposition   \ref{automorfismosNQ} the following maps

\begin{center}
\renewcommand{\arraystretch}{1.4}
\begin{tabular}{m{6cm} m{6cm}}
$\displaystyle
\varphi_1:\;
\begin{aligned}
x_1 &\longmapsto x_1[x_2,x_1]^{\frac{-\alpha_{3}}{\alpha_{1}} + \frac{\alpha_{1}-1}{2}},\\
x_2 &\longmapsto x_2^{\alpha_{1}}
\end{aligned}
$
&
$\displaystyle
\varphi_2:\;
\begin{aligned}
x_1 &\longmapsto x_2^{-1}[x_2,x_1]^{-\frac{\alpha_{2}}{\alpha_{1}} + \frac{\alpha_{1}-1}{2}},\\
x_2 &\longmapsto x_1^{\alpha_{1}}
\end{aligned}
$
\\[2em]
$\displaystyle
\varphi_3:\;
\begin{aligned}
x_1 &\longmapsto x_1^{\alpha_{1}}[x_2,x_1]^{-\alpha_{3}},\\
x_2 &\longmapsto x_1^{\frac{\alpha_{2}}{\alpha_{3}} - \frac{\alpha_{1}(\alpha_{1}-1)}{2\alpha_{3}}}\,x_2
\end{aligned}
$
&
$\displaystyle
\varphi_4:\;
\begin{aligned}
x_1 &\longmapsto x_2[x_2,x_1]^{- \frac{\alpha_{1}(\alpha_{1} - 1)}{2\alpha_{1}}},\\
x_2 &\longmapsto x_1^{-\alpha_{1}}
\end{aligned}
$
\end{tabular}
\end{center}
extend to automorphisms of $G$ and they are such that $h^{\varphi_1} = [x_1, x_2]$ in each case.
Therefore,
$$
G' \setminus Z(G) = [x_2, x_1]^{\Aut(G)}.
$$

Let $z \in Z(G) \setminus \{1\}$. Then
$$
z =
[x_2, x_1, x_1]^{\alpha_{1}}
[x_2, x_1, x_2]^{\alpha_{2}},
$$
where $(\alpha_{1}, \alpha_{2}) \in \mathbb{Q}^2 \setminus \{(0,0)\}$.
We consider the possible cases, and by Proposition \ref{automorfismosNQ}, the following maps
\begin{center}
\renewcommand{\arraystretch}{1.3}
\begin{tabular}{p{4cm} p{4cm} p{4cm}}

$\displaystyle
\phi_1:\;
\begin{aligned}
x_1 &\longmapsto x_1 \\
x_2 &\longmapsto x_2^{\alpha_{1}}
\end{aligned}
$
&
$\displaystyle
\phi_2:\;
\begin{aligned}
x_1 &\longmapsto x_1^{\alpha_{2}}\\
x_2 &\longmapsto x_2
\end{aligned}
$
&
$\displaystyle
\phi_3:\;
\begin{aligned}
x_1 &\longmapsto x_1 x_2^{\frac{\alpha_{2}}{\alpha_{1}}}\\
x_2 &\longmapsto x_2^{\alpha_{2}}
\end{aligned}
$

\end{tabular}
\end{center}
extend to automorphisms of $G$ and they are such that $z^{\phi_i} = [x_2, x_1, x_1], i \in \{1,2,3\}$.
Hence,
$$
Z(G) \setminus \{1\} = [x_2, x_1, x_1]^{\Aut(G)}.
$$
Therefore, $\omega(G) = 4$. 
\end{proof}

Now we will prove that, except for the cases $N_{r,2}^{\mathbb{Q}}$ and $N_{2,3}^{\mathbb{Q}}$, the group $N_{r,c}^{\mathbb{Q}}$ has an infinite number of orbits under automorphisms.

\begin{prop}
The group $G = N_{2,c}^{\mathbb{Q}}$ does not have finitely many automorphism orbits whenever $c>3$.
\end{prop}

\begin{proof}
We prove the case $N_{2,4}^{\mathbb{Q}}$, and the general case follows by induction.  
Let $N_{2,4} = \langle x_1, x_2 \rangle$. Thus,
$G = N_{2,4}^{\mathbb{Q}} = \langle x_1, x_2 \rangle^{\mathbb{Q}}$.
We order the basic commutators of weight $1$ as $x_1 < x_2$. Then the central basic commutators are
\[
[x_2, x_1, x_1, x_1], \ [x_2, x_1, x_1, x_2], \ \text{and } [x_2, x_1, x_2, x_2].
\]
Consider the map
\begin{center}
\renewcommand{\arraystretch}{1.5}
\begin{tabular}{r c l}
$\varphi: x_1$ & $\longmapsto$ & $x_1^{a}\, x_2^{b} g_1$ \\
$x_2$ & $\longmapsto$ & $x_1^{c}\, x_2^{d} g_2$
\end{tabular}
\end{center}
where $a,b,c,d \in \mathbb{Q}$ and $g_1, g_2 \in G'$. By Proposition~\ref{automorfismosNQ}, this map extends to an automorphism of $G$.

Let
\[
[x_2,x_1,x_1,x_1]^u \ [x_2,x_1,x_1,x_2]^v \ [x_2,x_1,x_2,x_2]^w \in Z(G),
\]
with $(u,v,w) \in \mathbb{Q}^3 \setminus \{(0,0,0)\}$.  
The action of $\varphi$ on the $3$-dimensional $\mathbb{Q}$-vector space $Z(G)$ is given by the matrix
\[
\begin{pmatrix}
a^2D & 0 & b^2D \\
acD & D^2 & bdD \\
c^2D & 0 & d^2D
\end{pmatrix}, \text{ where } D = ad - bc.
\]
In particular, the element $[x_2,x_1,x_1,x_2]^v$ is mapped by $\varphi$ to
\[
[x_2,x_1,x_1,x_1]^{acDv}
[x_2,x_1,x_1,x_2]^{D^2v}
[x_2,x_1,x_2,x_2]^{bdDv}.
\]

Suppose, by contradiction, that $G$ has finitely many automorphism orbits. Since there are infinitely many prime numbers and $\omega(G) < \infty$, there exist distinct primes $p,q$ and $\varphi \in \operatorname{Aut}(G)$ such that 

\[
([x_2,x_1,x_1,x_2]^p)^{\varphi}
=
[x_2,x_1,x_1,x_1]^{acDp}
[x_2,x_1,x_1,x_2]^{D^2p}
[x_2,x_1,x_2,x_2]^{bdDp}
= \]
\[
= [x_2,x_1,x_1,x_2]^q.
\]
Since $D \neq 0$ and the unique expression of elements in the basis 
\[\{[x_2,x_1,x_1,x_1]
[x_2,x_1,x_1,x_2]
[x_2,x_1,x_2,x_2]\}\]
of the $\mathbb{Q}$-vector space $Z(G)$ we may assume that either $a=c=0$ or $b=d=0$, which yields
\[
([x_2,x_1,x_1,x_2]^p)^{\varphi}
=
[x_2,x_1,x_1,x_2]^{D^2p}
=
[x_2,x_1,x_1,x_2]^q.
\]
Hence $D^2p = q$, that is, 
$D^2 = \frac{q}{p}$,
which is impossible for a $D$ in $\mathbb{Q}$. This contradiction shows that $G$ does not have finitely many automorphism orbits.

Since $N_{r,c-1} \simeq N_{r,c}/\gamma_c(N_{r,c})$, in particular we have
\[
N_{2,4}^{\mathbb{Q}} \simeq \frac{N_{2,5}^{\mathbb{Q}}}{Z(N_{2,5}^{\mathbb{Q}})}.
\]
As the group $N_{2,4}^{\mathbb{Q}}$ does not have finitely many automorphism orbits, the same holds for $N_{2,5}^{\mathbb{Q}}$. By induction, we conclude that the group $N_{2,c}^{\mathbb{Q}}$ does not have finitely many automorphism orbits for all $c \geq 4$.
\end{proof}
\begin{prop}\label{theorem4.4}
If $r \geq 3$ the group $N_{r,3}^{\mathbb{Q}}$ does not have finitely many automorphism orbits.
\end{prop}

\begin{proof}

Let $N_{r,3} = \langle x_1, x_2, x_3, \cdots, x_r \rangle$ and 
$G = \langle x_1, x_2, x_3, \cdots , x_r \rangle^{\mathbb{Q}}$.
We order the basic commutators of weight $1$ as $x_1 < x_2 < x_3 < \dots < x_{r}$. Therefore, the basic central commutators of $G$ are
$$[x_2, x_1, x_1],\ [x_2, x_1, x_2],\ [x_2, x_1, x_3],\ 
[x_3, x_1, x_1],\ [x_3, x_1, x_2],\ [x_3, x_1, x_3],$$
$$[x_3, x_2, x_2],\ [x_3, x_2, x_3], \dots, [x_{r}, x_{r-1}, x_{r}].$$
Suppose, by contradiction, that $G$ has finitely many automorphism orbits.  
Consider the element
\[
[x_2, x_1, x_1]^p [x_3, x_2, x_3] \in Z(G),
\]
where $p$ is a prime number. Since there are infinitely many primes, there must exist an automorphism $\varphi$ of $G$ such that
\[
([x_2, x_1, x_1]^p [x_3, x_2, x_3])^{\varphi}
=
[x_2, x_1, x_1]^q [x_3, x_2, x_3],
\]
where $p$ and $q$ are distinct primes.
Since $\varphi$ is an automorphism of $G$, by Proposition~\ref{automorfismosNQ}, $\varphi$ induces a map
\begin{center}
\renewcommand{\arraystretch}{1.5}
\begin{tabular}{r c l}
$\varphi : x_1$ & $\longmapsto$ & $x_1^{\alpha_{11}} x_2^{\alpha_{12}} x_3^{\alpha_{13}} \dots x_r^{\alpha_{1r}} f_1$ \\
$x_2$ & $\longmapsto$ & $x_1^{\alpha_{21}} x_2^{\alpha_{22}} x_3^{\alpha_{23}} \dots x_r^{\alpha_{2r}}f_2$ \\
$x_3$ & $\longmapsto$ & $x_1^{\alpha_{31}} x_2^{\alpha_{32}} x_3^{\alpha_{33}} \dots x_r^{\alpha_{3r}} f_3$\\
$\vdots$\\
$x_r$ & $\longmapsto$ & $x_1^{\alpha_{r1}} x_2^{\alpha_{r2}} x_3^{\alpha_{r3}} \dots x_r^{\alpha_{rr}} f_r$
\end{tabular}
\end{center}
such that
\[
\{(\alpha_{11},\alpha_{12}, \dots,\alpha_{1r}),(\alpha_{21},\alpha_{22},\dots, \alpha_{2r}),(\alpha_{31},\alpha_{32}, \dots, \alpha_{3r}) \dots (\alpha_{r1}, \alpha_{r2} \dots \alpha_{rr})\}
\]
is a basis of $\mathbb{Q}^r$ and $f_1,f_2,f_3, \dots, f_r \in G'$. Note that,
\[([x_2,x_1,x_1]^p [x_3,x_2,x_3])^{\varphi}
= z_1 z_2 z_3 z_4 z_5 z_6 z_7 z_8 \dots z_k
= [x_2,x_1,x_1]^q [x_3,x_2,x_3],\]
where $k$ denotes the number of basic commutators of weight $3$ and \\
$$\begin{aligned}
z_1 &= [x_2,x_1,x_1]^{p\alpha_{11}X_1 + \alpha_{31}Y_1},\\
z_2 &= [x_2,x_1,x_2]^{p\alpha_{12}X_1 + \alpha_{32}Y_1},\\
z_3 &= [x_2,x_1,x_3]^{p\alpha_{13}X_1 + \alpha_{33}Y_1
      - p\alpha_{11}X_3 - \alpha_{31}Y_3},\\
z_4 &= [x_3,x_1,x_1]^{p\alpha_{11}X_2 + \alpha_{31}Y_2},\\
z_5 &= [x_3,x_1,x_2]^{p\alpha_{12}X_2 + \alpha_{32}Y_2
      + p\alpha_{11}X_3 + \alpha_{31}Y_3},\\
z_6 &= [x_3,x_1,x_3]^{p\alpha_{13}X_2 + \alpha_{33}Y_2},\\
z_7 &= [x_3,x_2,x_2]^{p\alpha_{12}X_3 + \alpha_{32}Y_3},\\
z_8 &= [x_3,x_2,x_3]^{p\alpha_{13}X_3 + \alpha_{33}Y_3}\\
\vdots\\
z_k &= [x_r, x_{r-1}, x_r]^{p\alpha_{1r}X_l + \alpha_{3r}Y_l}
\end{aligned}$$ \\
where $X_1, X_2, X_3, Y_1, Y_2$ and $Y_3$ are give by
$$\begin{array}{c c c}
X_1 = \alpha_{22}\alpha_{11}-\alpha_{12}\alpha_{21} & 
X_2 = \alpha_{23}\alpha_{11}-\alpha_{21}\alpha_{13} & 
X_3 = \alpha_{23}\alpha_{12}-\alpha_{22}\alpha_{13} \\[0.8em]
Y_1 = \alpha_{32}\alpha_{21}-\alpha_{22}\alpha_{31} & 
Y_2 = \alpha_{33}\alpha_{21}-\alpha_{23}\alpha_{31} & 
Y_3 = \alpha_{33}\alpha_{22}-\alpha_{23}\alpha_{32} 
\end{array}$$
As the following set 
$$\{[x_2, x_1, x_1],\ [x_2, x_1, x_2],\ [x_2, x_1, x_3],\ 
[x_3, x_1, x_1],\ [x_3, x_1, x_2],\ [x_3, x_1, x_3],$$
$$[x_3, x_2, x_2],\ [x_3, x_2, x_3], \dots, [x_{r}, x_{r-1}, x_{r}]\},$$
is a base of the $\mathbb{Q}$-vector space $Z(G)$, we obtain the following nonlinear system

\vspace{5.0mm}

$$R \ \  
\left\{ \begin{array}{l@{\quad}r} 
{p\alpha_{11}X_1 + \alpha_{31}Y_1}= q &\text{(1)}\\
{p\alpha_{12}X_1 + \alpha_{32}Y_1} = 0 & \text{(2)}\\
{p\alpha_{13}X_1 + \alpha_{33}Y_1
      - p\alpha_{11}X_3 - \alpha_{31}Y_3} = 0 & \text{(3)}\\
{p\alpha_{11}X_2 + \alpha_{31}Y_2}=0 &\text{(4)}\\
{p\alpha_{12}X_2 + \alpha_{32}Y_2
      + p\alpha_{11}X_3 + \alpha_{31}Y_3}=0&\text{(5)}\\
{p\alpha_{13}X_2 + \alpha_{33}Y_2} = 0 & \text{(6)}\\
{p\alpha_{12}X_3 + \alpha_{32}Y_3} = 0 &\text{(7)}\\
{p\alpha_{13}X_3 + \alpha_{33}Y_3} = 1 &\text{(8)} \\
\vdots \\
{p\alpha_{1r}X_l + \alpha_{3r}Y_l} = 0 & \text{($k$)}
\end{array} \right.$$\\
Consider the subsystem $S$ consisting of the first eight equations of $R$.
$$S \ \  
\left\{ \begin{array}{l@{\quad}r} 
{p\alpha_{11}X_1 + \alpha_{31}Y_1}= q &\text{(1)}\\
{p\alpha_{12}X_1 + \alpha_{32}Y_1} = 0 & \text{(2)}\\
{p\alpha_{13}X_1 + \alpha_{33}Y_1
      - p\alpha_{11}X_3 - \alpha_{31}Y_3} = 0 & \text{(3)}\\
{p\alpha_{11}X_2 + \alpha_{31}Y_2}=0 &\text{(4)}\\
{p\alpha_{12}X_2 + \alpha_{32}Y_2
      + p\alpha_{11}X_3 + \alpha_{31}Y_3}=0&\text{(5)}\\
{p\alpha_{13}X_2 + \alpha_{33}Y_2} = 0 & \text{(6)}\\
{p\alpha_{12}X_3 + \alpha_{32}Y_3} = 0 &\text{(7)}\\
{p\alpha_{13}X_3 + \alpha_{33}Y_3} = 1 &\text{(8)}\end{array} \right.$$\\
Note that, if \( R \) admits a solution, then this subsystem also admits a solution. 
We will show that \( R \) has no solution in \( \mathbb{Q}^9 \) by proving that the system $S$ admits no solution. 
To this end, we begin by analyzing the subsystem \( S_1 \), formed by equations \( (4) \) and \( (6) \).
\begin{center}
$ S_1 \ \  \left\{ \begin{array}{l@{\quad}r}

       {p \alpha_{11}X_2 + \alpha_{31}Y_2 = 0} & \text{(4)}\\
               
                 {p \alpha_{13}X_2 +  \alpha_{33}Y_2 = 0.} & \text{(6)}\\
\end{array} \right.$           
\end{center}
Let $M = \alpha_{11}\alpha_{33} - \alpha_{31}\alpha_{13}$. We have two cases: $M \neq 0$ or $M = 0$.\\
To prove the case $M \neq 0$, we first show that  $D = \alpha_{11}\alpha_{32} - \alpha_{12}\alpha_{31} \neq 0$. To this end, we consider the subsystem consisting of equations   $(1)$ and $(2)$:
\begin{center}
$ \left\{ \begin{array}{l@{\quad}r} 

 {p \alpha_{11}X_1 + \alpha_{31}Y_1 = q }& \text{(1)}\\
    {p \alpha_{12}X_1 + \alpha_{32}Y_2 = 0} & \text{(2)}
\end{array} \right.$  
\end{center}
Suppose, for the sake of contradiction, that $D = 0$. Note that $D$ is the determinant of the matrix
\[
\begin{pmatrix}
\alpha_{11} & \alpha_{12} \\
\alpha_{31} & \alpha_{32}
\end{pmatrix} 
\]
We claim that the vector $(\alpha_{11}, \alpha_{12}) \in \mathbb{Q}^2 \setminus \{(0,0)\}$. Indeed, if $\alpha_{11} = \alpha_{12} = 0$, then from equations $(1)$ and $(2)$ we obtain $\alpha_{32} = 0$. Thus, equation $(5)$ becomes $\alpha_{31}\alpha_{33}\alpha_{22} = 0$. We conclude that $\alpha_{33} = 0$, since from equation $(1)$ we have $\alpha_{22} \neq 0$ and $\alpha_{31} \neq 0$.
Note that from $(1)$ we have $\alpha_{22} = -\frac{q}{\alpha_{31}^2}$, and from $(8)$ we have $\alpha_{22} = -\frac{1}{p\alpha_{13}^2}$. Equating these expressions, we obtain  $qp = \frac{\alpha_{31}^2}{\alpha_{13}^2}$, which is impossible in $\mathbb{Q}$.
\vspace{0.3cm}
Hence, there exists $r \in \mathbb{Q} \setminus \{0\}$ such that 
$(\alpha_{11}, \alpha_{12}) = r(\alpha_{31}, \alpha_{32})$, that is, 
$\alpha_{11} = r\alpha_{31}$ and $\alpha_{12} = r\alpha_{32}$. Substituting these equalities into the equations $(1)$ and $(2)$, we obtain, respectively,
$$\underbrace{(\alpha_{22} \alpha_{31} - \alpha_{32} \alpha_{21})}_{\neq 0}\underbrace{(pr^2 \alpha_{31} - \alpha_{31})}_{\neq 0} = q \text{ and }$$
$$\underbrace{(\alpha_{22} \alpha_{31} - \alpha_{32}\alpha_{21})}_{\neq 0}(pr^2 \alpha_{32} - \alpha_{32})=0.$$
Thus, if $\alpha_{32} \neq 0$, then $pr^2 \alpha_{32} = \alpha_{32}$, and therefore $pr^2 = 1$. Hence, $r^2 = \frac{1}{p}$, which is impossible in $\mathbb{Q}$. It follows that $\alpha_{32} = 0$, which implies $\alpha_{12} = 0$, since $(\alpha_{11}, \alpha_{12}) = r(\alpha_{31}, \alpha_{32})$, from which we conclude that $\alpha_{11} \neq 0$ and $\alpha_{31} \neq 0$.
In a similar way, considering equations $(7)$ and $(8)$, it follows that $\alpha_{13} \neq 0$ and $\alpha_{33} \neq 0$. Since $\alpha_{12} = \alpha_{32} = 0$, it is immediate from equations $(1)$ and $(8)$ that $\alpha_{22} \neq 0$.
Therefore, the following equations must be satisfied simultaneously.
\vspace{0.3cm}

{\footnotesize
$$\left\{ \begin{array}{l@{\quad}r} 
{p \alpha_{11}^2\alpha_{22}  - \alpha_{31}^2 \alpha_{22}  = q } & \text{(1')}\\
          { p\alpha_{11}\alpha_{13} -\alpha_{33}\alpha_{31} =0} & \text{(3')}\\
       {p \alpha_{11}(\alpha_{23} \alpha_{11} -\alpha_{21} \alpha_{13}) + \alpha_{31}(\alpha_{33} \alpha_{21} -\alpha_{23} \alpha_{31}) = 0} & \text{(4')}\\
               { p \alpha_{11}(- \alpha_{22} \alpha_{13}) + 
                \alpha_{31} (\alpha_{33} \alpha_{22} ) = 0} & \text{(5')}\\
                {p \alpha_{13} (\alpha_{23} \alpha_{11} -\alpha_{21} \alpha_{13}) +  \alpha_{33}(\alpha_{33} \alpha_{21} -\alpha_{23} \alpha_{31}) = 0} & \text{(6')}\\        {-p \alpha_{13}^2 \alpha_{22} + \alpha_{33}^2 \alpha_{22}  = 1} & \text{(8')} \\
\end{array} \right.$$}\\
From equation $(1')$, we have
$\alpha_{22} (p \alpha_{11}^2 - \alpha_{31}^2) = q$.
On the other hand, from equation $(8')$, we obtain
$\alpha_{22} = \frac{1}{-p \alpha_{13}^2 + \alpha_{33}^2}$.
Substituting this expression into the first equation, it follows that\[
p \alpha_{11}^2 - \alpha_{31}^2 = q \, (\alpha_{33}^2 - p \alpha_{13}^2).
\]
From equation $(3')$, we have
$\alpha_{31} = \frac{p \, \alpha_{11} \alpha_{13}}{\alpha_{33}}$.
Replacing this value in the previous equality, we deduce
$\alpha_{33}^2 = \frac{q}{p}$,
which is absurd in $\mathbb{Q}$. Therefore, we conclude that $D \neq 0$. \\

\noindent \textbf{Case 1. $M \neq 0$.} \\
Since $M \neq 0$, we have $\alpha_{11} \alpha_{33} \neq \alpha_{31} \alpha_{13}$. Solving the above subsystem $S_1$, 
we obtain $X_2 = (\alpha_{23}\alpha_{11} - \alpha_{21}\alpha_{13}) = Y_2 = (\alpha_{33}\alpha_{21} - \alpha_{23} \alpha_{31}) = 0$.\\
With this, we obtain the following system consisting of equations $(5), (7)$, and $(8)$

\vspace{0.3cm}

$$\left\{
\begin{array}{ll}
p \alpha_{11}X_3 + \alpha_{31} Y_3 = 0 
\\
p \alpha_{12}X_3 + \alpha_{32}Y_3=0
\\
p\alpha_{13}X_3 + \alpha_{33} Y_3= 1

\end{array}
\right.$$

\vspace{0.3cm}

Thus,
$$X_3 = - \frac{\alpha_{31}Y_3}{p \alpha_{11}} \  \text{e} \  Y_3 (\underbrace{\alpha_{32} \alpha_{11} - \alpha_{12} \alpha_{31}}_{D \neq 0}) = 0 \Rightarrow X_3 = Y_3 = 0.$$
And the last equation reduces to the absurdity $0 = 1$. Therefore, this case does not occur.\\

\noindent \textbf{Case 2. $M = 0$}\\ 
Since $M = 0$, we have $\alpha_{11} \alpha_{33} = \alpha_{31} \alpha_{13}$.
Note that $M$ is the determinant of the matrix
\[
\begin{pmatrix}
    \alpha_{11} & \alpha_{31}\\
    \alpha_{13} & \alpha_{33}
\end{pmatrix}
\]
Therefore, there exists $s \in \mathbb{Q}$ such that $(\alpha_{11}, \alpha_{31}) = s(\alpha_{13}, \alpha_{33})$, that is, $\alpha_{11} = s \alpha_{13}$ and $\alpha_{31} = s \alpha_{33}$. Substituting this information into $(4)$, we have
$$(\alpha_{23}s - \alpha_{21})(ps\alpha_{13}^2 - s\alpha_{33}^2)=0.$$
If $\alpha_{21} = s \alpha_{23}$, then the first and third columns of the matrix
\[
\begin{pmatrix}
    \alpha_{11} & \alpha_{12} & \alpha_{13}\\
     \alpha_{21} & \alpha_{22} & \alpha_{23}\\
    \alpha_{31} & \alpha_{32} & \alpha_{33}
\end{pmatrix}
\]
are multiples and $\varphi$ is not an automorphism, which is an absurd. Therefore, $p s \alpha_{13}^2 - s \alpha_{33}^2 = 0$. Then, $p = \frac{\alpha_{33}^2}{\alpha_{13}^2}$, another contradiction. Thus, we cannot have $M = 0$. Hence, the system $S$ has no solution in $\mathbb{Q}^9$.
\end{proof}
We obtain the main result.

\begin{thmC}
The Mal'cev $\mathbb{Q}$-completion of the free nilpotent $r$-generated group $N_{r, c}$ of class $c$ has finitely many orbits under automorphisms if, and only if either $c \leq 2$, or $r = 2$ and $c = 3$. 
\end{thmC}

\bigskip

The authors would like to thank Raimundo de Araújo Bastos Júnior for the fruitful conversations about Theorem B.


\begin{thebibliography}{99}


\bibitem{N_0} A. B. Apps, \textit{On the Structure of $\aleph_0$-Categorical Groups}, J. Algebra,  81 (1983), 320-339.

\bibitem{BP} A. B. Apps, \textit{Boolean powers of groups}, Math. Proc. Camb. Phil. Soc., 91 (1982), 375 - 396.



\bibitem{verde} A. E. Clement, S. Majewicz, and M. Zyman, {\it The theory of nilpotent groups}, Springer, (2017).

\bibitem{BD} Bastos R. A. and A. C Dantas, {\it FC-groups with finitely many automorphism orbits}, J. Algebra, 516 (2018) 401-413.

\bibitem{BDdM} Bastos R. A. and A. C. Dantas, and E. de Melo, \textit{Virtually nilpotent groups with finitely many orbits under automorphisms}, Arch. Math. 116 (2021), 261–270.

\bibitem{Osin} D. Osin, {\it Small cancellations over relatively hyperbolic groups and embedding theorems}, Annals of Mathematics,  172 (2010)  1-39. 





\bibitem{J} E. Jabara, \emph{Abelian automorphisms groups with a finite number of orbits}, J. Algebra, 323 (2010), 2798–2817


\bibitem{R} J. G. Rosenstein, \emph{$\aleph_{0}$-Categoricity of Groups}, J. Algebra, 25 (1973), 435-467.

\bibitem{MS2} M. Schwachhöfer and M. Stroppel, Finding representatives for the orbits under the automorphism group of a bounded abelian group, J. Algebra 211 (1) (1999), 225-239.

\bibitem{S1} M. Stroppel, \emph{Locally compact groups with many automorphisms}, J. Group Theory 4 (2001), 427-455

\bibitem{S2} M. Stroppel, \textit{Locally compact groups with few orbits under automorphisms}, Top. Proc.,  26 (2) (2002), 819-842.

\bibitem{Hall} M. Hall, The Theory of groups, Macmillan, New York, 1959. 

\bibitem{M} O. Macedo\'nska, {\it On difficult problems and locally graded groups}, J. Math. Sci. (N.Y.), 142 (2007) 1949-1953.























 







\bibitem{larguracomutator} V. A. Roman'kov, {\it The commutator width of some relatively free lie algebras
and nilpotent groups}. Siberian Mathematical Journal, (2016), 679–695.

 
\end{thebibliography}
\end{document}